\documentclass[12pt,leqno,a4paper]{article}
\usepackage{hyperref}
\usepackage{amsmath}
\usepackage{amssymb}
\usepackage{amsthm}
\usepackage{graphicx,epstopdf}

\hypersetup{
    colorlinks=true,
    linkcolor=black,
    filecolor=black,      
    urlcolor=blue,
    citecolor=red,
}

\makeatletter
\newcommand{\guio}[1]{\nobreakdash-\hspace{0pt}}

\newtheorem*{theorem*}{Theorem}

\newtheorem*{lemmam*}{Lemma 5. (\cite{MOV})}
\newtheorem*{lemma*}{Lemma}
\newtheorem*{corollary*}{Corollary}

\newtheorem{remark*}{Remark}

\theoremstyle{definition}

\newtheorem*{acknowledgements}{Acknowledgements}

\newcommand{\R}{\mathbb{R}}
\newcommand{\C}{\mathbb{C}}

\newcommand{\re}{\operatorname{Re}}
\newcommand{\pv}{\operatorname{p.v.}}

\newcommand{\ep}{\varepsilon}
\newcommand{\e}{\varepsilon}

\newcommand{\ba}{\overline{z}}

\title{Explicit minimisers of some nonlocal anisotropic energies: a short proof}

\author{J.\ Mateu, M.~G.\ Mora, L.\ Rondi, L.\ Scardia and J.\ Verdera}

\date{}

\begin{document}
\maketitle
\begin{abstract}
In this paper we consider nonlocal energies defined on probability measures in the plane, given by a convolution interaction term 
plus a quadratic confinement. The interaction kernel is $-\log|z|+\alpha\, x^2/|z|^2, \; z=x+iy,$ with $-1 < \alpha< 1.$ 
This kernel is anisotropic except for the Coulombic case $\alpha=0.$  We present a short compact proof of the known surprising fact 
that the unique minimiser of the energy is the  normalised characteristic function of the domain enclosed by an ellipse with horizontal
semi-axis $\sqrt{1-\alpha}$
and vertical semi-axis $\sqrt{1+\alpha}.$ Letting $\alpha \to 1^-$ we find that the semicircle law on the vertical axis is the unique 
minimiser of the corresponding energy, a result related to interacting dislocations, and previously obtained by some of the authors. 
We devote the first sections of this paper to presenting some well-known background material in the simplest way possible, so that
readers unfamiliar with the subject find the 
proofs accessible.
\bigskip


\medskip

\end{abstract}

\section{Introduction}

Consider the energy functional defined on a  probability measure $\mu$ in the plane by
\begin{equation}\label{en}
I_\alpha(\mu)=\iint W_\alpha(z-w) \,d\mu(z) \,d\mu(w) + \int |z|^2 \,d\mu(z),
\end{equation}
where the interaction kernel is
\begin{equation}\label{intker}
W_\alpha(z) = -\log|z|+\alpha \frac{x^2}{|z|^2}, \quad z=x+i y \in \C, \quad z\neq0, \quad  \alpha \in \R.
\end{equation}
 The result we discuss here was proved in \cite{CMMRSV} and reads as follows.
 \begin{theorem*}\label{Theorem}
 If $-1< \alpha < 1,$ then the unique minimiser of the energy functional $I_\alpha$ is the normalised characteristic function of the domain enclosed by the ellipse (centred at zero) with horizontal
 semi-axis $\sqrt{1-\alpha}$ and vertical semi-axis $\sqrt{1+\alpha}.$
 \end{theorem*}
 
For $\alpha=0$ this result is already proved in Frostman's thesis \cite{F}.  A simple argument shows that as $\alpha \to 1^-$  
the minimiser  in the Theorem tends in the weak $\star$ topology of finite Radon measures to the semi-circle law on the vertical axis, namely, to the 
probability measure
\begin{equation*}\label{sclaw}
\frac{1}{\pi} \sqrt{2-y^2}\;\chi_{[-\sqrt{2}i,\sqrt{2}i]}(y)\,dy.
\end{equation*}
One can show that indeed the semi-circle law on the vertical axis is the unique minimiser of the energy \eqref{en} with $\alpha=1.$ This was proved
in \cite{MRS}, and solved a long standing conjecture on the behaviour of interacting dislocations in metals, 
which predicted the formation of walls. Swapping variables one gets a similar result for $\alpha=-1$ involving the semicircle law on the horizontal axis. An argument based on energy comparison then leads to the conclusion that for $|\alpha| >1$
the unique minimiser of~\eqref{en} is one of the semicircle laws, which completes the picture.

Energies of the form \eqref{en}, with various types of interaction kernels, arise frequently in models in which individuals
repel each other if they get too close and are attracted if they get far from the centre of mass. The work done so far has concentrated mostly on radial interaction kernels and only very recently non-isotropy has entered the scene. See the introductions of \cite{MRS}, \cite{CMMRSV}, \cite{CMMRSV2} and \cite{MMRSV} for more information about that and for references to previous work on the subject. 

The main goal of this paper is to present a short compact proof of the Theorem. 
Along the way we also present some well-known background results, for the sake of the reader unfamiliar with the 
subject. The original proof in \cite{CMMRSV} relies on computing explicitly, at every point in the plane, the potential $P$ 
of the normalised characteristic function of the compact set $E$  enclosed by a generic ellipse (see \eqref{potmu} below for the 
definition of
the potential of a finite Radon measure $\mu$). With the full potential at hand, it is then shown that there exists a unique ellipse
such that the corresponding potential satisfies the so-called Euler-Lagrange conditions (the first condition on $E$, 
the second outside $E$), which are necessary conditions for minimality. This ellipse is shown to provide a minimiser of 
$I_\alpha,$  which is unique by the strict convexity of the energy.
This computational approach is very powerful, but it does not provide much insight about the 
deep nature of the problem.  In the argument we present here one needs instead to know explicitly the potential $P$ only inside ellipses,
which is much easier. 
The explicit formula for the potential inside the ellipse yields appropriate values for the semi-axes so that the potential is
constant on $E.$
In other words, it leads to finding a solution to the first Euler--Lagrange condition. It remains to show that $P$ is everywhere 
larger than or equal to the constant value it takes on $E$, which is the second Euler--Lagrange condition. We achieve that in two steps. The first one consists in showing that the Laplacian of $P$ on the exterior of $E$ has positive boundary limits. This we do via
the classical Plemelj  jump formula for the Cauchy Integral. The second step is an application of the minimum principle to a suitably constructed function, which exploits the fact that $P$ is biharmonic outside $E$. 

The short compact proof we present here, however, does not generalise to dimensions higher than two, unlike the original computational 
approach, which was exploited in \cite{CMMRSV2} to prove the natural higher-dimensional version of the result in \cite{CMMRSV}. Moreover, other energies in the plane involving other natural interaction kernels do not seem to be covered by what we do in this paper. 

In conclusion, the problem of understanding deeply why ellipses appear in minimising the energy \eqref{en} for interaction kernels with a similar structure to \eqref{intker} turns out to be challenging, and remains at present rather obscure. Further work seems to be needed to unveil its real nature.

The paper is organised as follows. Sections 2, 3 and 4 are expository and aimed at readers unfamiliar with the subject. 
We discuss some properties of the potentials we will be dealing with, the Euler--Lagrange conditions for the energy minimisers, and existence
and uniqueness of minimisers. Sections 5 and 6 contain the proof 
of the Theorem. In section 5 we find a candidate ellipse such that the potential of the normalised characteristic function of the interior domain satisfies the first Euler--Lagrange condition. In section 6 we prove that such potential satisfies the second Euler--Lagrange condition. Section 7 is an appendix devoted to the Plemelj jump formula.

\section{The potential}
Given a mass  distribution $\mu$, one can define a potential of $\mu$ associated with the energy \eqref{en}. This potential arises in computing the 
directional derivative of $I_\alpha$ at $\mu$ along a measure $\delta$ in the space of finite Radon measures (not necessarily positive, not necessarily with total mass $1$) with finite energy 
\begin{equation}\label{finen-0}
I_\alpha(|\delta|) <\infty.
\end{equation}
First of all, since the interaction kernel $W_\alpha$ in \eqref{intker} is even,
\begin{equation}\label{derpotmu}
\frac{d}{dt} I_\alpha(\mu+t\delta)\big|_{t=0} = 2 \int \left( \int W_\alpha(z-w) \,d\mu(w) + \frac{|z|^2}{2}\right)  \, d\delta(z),
\end{equation}
and so the following expression, which one calls the potential of $\mu,$ arises naturally:
\begin{equation}\label{potmu}
P(\mu)(z) = \left( W_\alpha \star \mu \right) (z)+ \frac{1}{2}|z|^2, \quad z
\in \mathbb{\C}.
\end{equation}
We claim that if $\mu$ is a probability measure minimising $I_\alpha$ then 
\begin{equation}\label{potmucons}
P(\mu)(z) = C_0, \quad , \quad \mu-a.e. \;\, \text{on}\;\,  \operatorname{spt}\mu,
\end{equation}
where $\operatorname{spt}\mu$ stands for the support of $\mu$ and $C_0$ is a constant.
Set $C_0= \int P(\mu) \, d\mu$ 
 and  define 
$$F_1= \{z\in \operatorname{spt}\mu: P(\mu)(z) \le C_0 \} \qquad \text{and} \qquad  F_2= \{z\in \operatorname{spt}\mu: P(\mu)(z) > C_0 \}.$$
Assume that $\mu(F_2)>0.$ If one has $\mu(F_1)=0$ then $C_0 = \int P(\mu) \,d\mu > C_0.$ Thus $\mu(F_1)>0.$
The measure
$$
\delta = \frac{1}{\mu(F_1)} \chi_{F_1} \mu - \frac{1}{\mu(F_2)} \chi_{F_2} \mu
$$
has total integral $0$ and finite energy \eqref{finen-0}.  Furthermore $\mu+t \delta$ is a positive measure with total mass $1$,  provided 
$$
-\mu(F_1) <  t < \mu(F_2).
$$
Note that the derivative in \eqref{derpotmu} vanishes, because $t=0$ is a minimum of $I_\alpha(\mu+t\delta)$. The right-hand side of \eqref{derpotmu}, however, is 
$$2\left(\frac{1}{\mu(F_1)} \int_{F_1} P(\mu) \,d\mu -\frac{1}{\mu(F_2)} \int_{F_2} P(\mu) \,d\mu \right) < 0,$$  
which is a contradiction. Therefore $\mu(F_2)=0$ and so $P(\mu)(z) \le C_0, \;\mu$-a.e.\ on $\operatorname{spt}\mu.$ 
Since $\int P(\mu)\,d\mu= C_0$ we obtain  \eqref{potmucons}. 

The argument above holds for kernels much more general than $W_\alpha$.

Moreover, for the kernel $W_\alpha$ in \eqref{intker}, one has that
$$P(\mu)(z) \le C_0, \quad z \in \operatorname{spt}\mu. $$
This can be proved readily as follows.  Since $\mu$ has finite $I_\alpha$ energy, its logarithmic energy is finite too.
Hence $\mu$ has no atoms
and $(x^2/|z|^2) \ast \mu$ is a continuous function on the plane. Therefore $P(\mu)$ is lower semicontinuous.
The set  $\{z\in \C: P(\mu)(z) > C_0\}$ is open and has zero $\mu$ measure 
and so does not intersect the support of $\mu.$

\smallskip

\noindent
As a final remark note that the definition \eqref{potmu} yields
$$
\int P(\mu)\, d\mu = I_\alpha(\mu) - \frac{1}{2} \int |z|^2\,d\mu(z).
$$
If one wants that, in analogy to what happens in electrostatics, $I_\alpha(\mu)= \int P(\mu)\, d\mu,$ 
then one has to add the constant $\frac{1}{2} \int |z|^2\,d\mu(z)$ to the right hand side of \eqref{potmu}. The resulting 
potential will also be $\mu$-\emph{a.e.}  constant on the support of $\mu$ if $\mu$ is a minimiser of $I_\alpha$.

\section{The Euler--Lagrange conditions}
The potential $P$ of a minimiser $\mu$ satisfies two properties, called the Euler--Lagrange conditions.
Let $Cap$ stand for the logarithmic capacity.
The Euler--Lagrange conditions {\bf{EL1}} and {\bf{EL2}} are the following.

\vspace{0.2cm}
{\bf{EL1}}: {\it{There exists a constant $C_0$ such that $P(\mu)(z)=C_0, \; \operatorname{Cap} $-a.e.\ $\text{on }\; 
\operatorname{spt}\mu$.}}

\vspace{0.2cm}

Before proceeding to the proof of the statement above we make a remark. Let $\nu$ be a probability measure with finite energy and set
 $\delta = \nu-\mu.$ Then $\mu +t\delta$
is a probability measure with finite energy for $0\le t \le 1.$ Since the function of $t \rightarrow I_\alpha(\mu+t \delta)$ has
a minimum at $t=0$ we conclude from \eqref{derpotmu} that
\begin{equation}\label{munu}
 \int P(\mu) \,d\nu \ge \int P(\mu)\, d\mu.
\end{equation}

We now prove \textbf{EL1}.
If the set $\{z \in \operatorname{spt}\mu : P(\mu) < C_0\}$ has positive capacity, then there exists a probability measure $\nu$ supported
on that set with finite logarithmic energy and thus with finite $I_\alpha$ energy. 
Now the right-hand side in \eqref{munu} is $C_0$ and the left-hand side is strictly less than $C_0,$ which is a contradiction.

\medskip


\vspace{0.2cm}

{\bf{EL2}}: {\it{$P(\mu)(z) \ge C_0, \quad z \notin \operatorname{spt} \mu, \; C_0$  the constant in {\bf{EL1}}}}.

\vspace{0.2cm}

The argument for proving {\bf EL2} is rather simple. 
Given a point $z\notin \operatorname{spt}\mu$, we set
$$
\nu=\frac{1}{|B(z,r)|} \chi_{B(z,r)}(w)\, dA(w),
$$
where $dA$ is Lebesgue measure in the plane. By \eqref{munu}
$$
C_0 = \int P(\mu)\,d\mu \le \frac{1}{|B(z,r)|} \int_{B(z,r)} P(\mu)(w) \,dA(w)  \xrightarrow{r \to 0} P(\mu)(z),
$$
since $P(\mu)$ is continuous on the complement of the support of $\mu.$ Hence {\bf EL2} holds.

\begin{remark*}[Sufficiency of the Euler-Lagrange conditions]
In the case of the functional $I_\alpha$ in \eqref{en} the Euler-Lagrange conditions are not only necessary conditions for minimisers, 
but they are also sufficient. In other words, they characterise minimisers of $I_\alpha$. This follows from the strict convexity of 
$I_\alpha$, which is proved in the next section. 
\end{remark*}

\section{Existence and uniqueness of minimisers}\label{EUm}
The existence of a minimiser for $I_\alpha$ follows by a standard compactness argument based on the lower semicontinuity and the
coercivity of the interaction kernel. In addition a minimiser has compact support since the quadratic confinement beats the 
interaction potential at infinity. 
Uniqueness follows from the fact that the Fourier transform of the interaction kernel $W_\alpha$ in \eqref{intker} is non-negative on 
test functions with zero integral. All this can be found in \cite{MRS}.  Here we discuss some of the steps in the computation of the Fourier 
transform of $W_\alpha$ and the proof of the uniqueness of minimisers from positivity of the Fourier transform of $W_\alpha$.

The definition of the Fourier transform  we use is
$$
\widehat{\phi}(\xi) = \int \phi(z) e^{-i \xi \cdot z}\, d A(z), \quad \xi \in \C, 
$$
where $\phi$ is a function in the Schwartz class $\mathcal{S}.$  The Fourier transform of the logarithmic term is
\begin{equation*}\label{}
 -\log|z| \xrightarrow{\rm{Fourier}}  \;\frac{2\pi}{|\xi|^2}+ c_0 \delta_0,
\end{equation*}
where $c_0$ is a constant, and the tempered distribution $1/|\xi|^2$ acts on $\varphi \in \mathcal{S}$ as
\begin{equation}\label{dis}
\int_{|\xi|<1} \frac{\varphi(\xi)-\varphi(0)}{|\xi|^2} \, d A(\xi) + \int_{|\xi|>1} \frac{\varphi(\xi)}{|\xi|^2} \, d A(\xi).
\end{equation}
To compute the Fourier transform of the anisotropic term in $W_\alpha$ we write
\begin{equation*}\label{ani}
\frac{x^2}{|z|^2} = \frac{1}{2} \left(\frac{x^2-y^2}{|z|^2}+1\right).
\end{equation*}
Since the homogeneous polynomial $x^2-y^2$ is harmonic we may resort to the well-known formula  \cite[Chapter 3, Section 3, p.73]{S}
\begin{equation*}\label{}
 \frac{x^2-y^2}{|z|^2} \xrightarrow{\rm{Fourier}} -4\pi \, \operatorname{p.v.} \frac{\xi_1^2-\xi_2^2}{|\xi|^4},
\end{equation*}
where $\operatorname{p.v.}$ stands for principal value.
Hence, for a constant $c_\alpha$ depending only on $\alpha,$
\begin{equation}\label{FW}
\begin{split}
W_\alpha \xrightarrow{\rm{Fourier}} & \;\frac{2\pi}{|\xi|^2}- \frac{\alpha}{2} \left(4\pi \, 
\operatorname{p.v.} \frac{\xi_1^2-\xi_2^2}{|\xi|^4}\right) + c_\alpha  \delta_0 \\*[10pt]
& \hspace{-0.4cm}= 2\pi \frac{(1-\alpha) \xi_1^2 + (1+\alpha) \xi_2^2}{|\xi|^4}+ c_\alpha \delta_0,
\end{split}
\end{equation}
and the tempered distribution of homogeneity $-2$ in the last line acts on  $\mathcal{S}$ in a way analogous to \eqref{dis}. 
From \eqref{FW} and by Plancherel's identity one gets
\begin{equation*}\label{}
(2\pi)^2 \iint W_\alpha (z-w) \phi(z) \phi(w) \,dA(z)dA(w) = \int \widehat{W_\alpha }(\xi) |\widehat{\phi}(\xi)|^2 \, dA(\xi) \geq 0,
\end{equation*}
provided $\phi$ is a function in $\mathcal{S}$ with vanishing integral. 

We now extend the above formula to a more general context. 

\begin{lemma*}\label{Lemma}
If $\mu_1$ and $\mu_2$ are compactly supported probability measures with finite energy, then
\begin{equation}\label{planch}
(2\pi)^2 \iint W_\alpha (z-w) \,d(\mu_1-\mu_2)(z) \, d(\mu_1-\mu_2)(w)= \int \widehat{W_\alpha }(\xi) |\widehat{\mu_1-\mu_2}(\xi)|^2 \, dA(\xi).
\end{equation}
\end{lemma*}

\begin{remark*}
The case $\alpha=0$ in \eqref{planch} corresponds to a purely logarithmic potential and is well-known. 
See, for example, \cite[Lemma 1.8, p.29]{ST} and \cite[Theorem 1.16, p.80]{L}, where it is shown that the left-hand side is
non-negative. Indeed, in both references one proves that
$$
\iint \log \frac{1}{|z-w|} \,d\nu(z) \, d\nu(w) = \int \left(\int \frac{1}{|z-w|}\,d\nu(w)\right)^2 \,dA(z),
$$
provided $\nu$ is a signed compactly supported  measure with finite logarithmic energy and $\nu(1)=0.$
See \cite[formula (1.30), p.34]{ST} and \cite[p.80]{L}. Applying Plancherel's identity to the right hand side above
one gets \eqref{planch} for $\alpha=0.$
\end{remark*}

Before proving the Lemma we show that \eqref{planch} implies uniqueness of the minimisers of $I_\alpha.$  We recall that at
the beginning of this section we observed that a minimiser has compact support. Hence we can assume that we are miminising over compactly
supported probability measures.
First of all, since the right-hand side in \eqref{planch} is non-negative, the Lemma yields
\begin{equation}\label{aaaa}
\begin{split}
2 \iint W_\alpha (z-w) \,d\mu_1(z) \,d\mu_2(w) &\le   \iint W_\alpha (z-w) \,d\mu_1(z) \,d\mu_1(w)\\
&\quad+ \iint W_\alpha (z-w) \,d\mu_2(z) \,d\mu_2(w),
\end{split}
\end{equation}
with strict inequality unless $\mu_1=\mu_2.$

Now the strict convexity of the energy functional $I_\alpha$ follows immediately. Note that, since the confinement term is linear one just needs to look at the energy associated with the interaction kernel, namely,
\begin{equation*}\label{}
J_\alpha(\mu)= \iint W_\alpha(z-w)\, d\mu(z)\,d\mu(w).
\end{equation*}
Now, if $0\le t\le1$ and $\mu_1$ and $\mu_2$ are probability measures with $J_\alpha(\mu_j) < \infty$, $j=1,2,$ then
\begin{equation*}\label{}
\begin{split}
J_\alpha((1-t)\mu_1+ t \mu_2)&= (1-t)^2 J_\alpha(\mu_1) +t^2 J_\alpha(\mu_2)+2t(1-t) \iint W_\alpha(z-w) \,d\mu_1(z)\,d\mu_2(w) \\*[10pt]
& \stackrel{\eqref{aaaa}}{\leq} (1-t)^2 J_\alpha(\mu_1) +t^2 J_\alpha(\mu_2)+t(1-t) \left(J_\alpha(\mu_1) +J_\alpha(\mu_2)\right)\\*[10pt]
& =(1-t) J_\alpha(\mu_1) +t J_\alpha(\mu_2)
\end{split}
\end{equation*}
and the inequality is strict unless $\mu_1=\mu_2.$ 

Of course strict convexity implies uniqueness of minimisers.

\begin{proof}[Proof of the Lemma] We follow closely the argument in \cite{CMMRSV2}.

Let $\varphi$ be a $C^{\infty}$ function, supported on the unit disc $B_1(0) \subset \R^2$, non-negative, radial, and with $\int_{\R^2} \varphi (z)\,dz =1.$  Set $\nu=\mu_1-\mu_2. $ For $\e>0$  we define
\begin{equation*}\label{convol}
\varphi_\e(z)= \frac{1}{\e^2} \varphi\left(\frac{z}{\e}\right) \qquad  \text{ and} \qquad \nu_\e= \nu \star \varphi_\e.
\end{equation*}
We claim that
\begin{equation}\label{parep}
(2\pi)^2 \int_{\R^2} (W_\alpha\star \nu_\e)(z) \nu_\e(z) \,dA(z) = \int_{\R^2} \widehat{W}_\alpha(\xi) |\widehat{\nu_\e}(\xi)|^2 \, dA(\xi).
\end{equation}
To show this, we set $f:=W_\alpha \star \nu_\e$ and $g:=\nu_\e$, and note that $g \in C^{\infty}_c(\R^2)$ and $f \in C^{\infty}(\R^2)$.
Moreover, since $\widehat{\nu_\e}\in{\mathcal S},$  $\widehat{\nu_\e}(0)=0$ and $\widehat{W}_\alpha$
behaves as $1/|\xi|^2$ at infinity in view of \eqref{FW}, we have that $\widehat{f}=\widehat{W}_\alpha\,\widehat{\nu_\e}\in L^1(\R^2)$.
Let $\psi \in C^{\infty}_c(\R^2)$ be such that $\psi = 1$ on~$B_1(0)$ and 
let $R>0$ be such that the support of $g$ is contained in $B_R(0)$. If $\tau>0$ is such that $\tau R < 1$, then,
by Parseval's formula,
\begin{equation}\label{lime}
\begin{split}
(2\pi)^2 \int_{\R^2} f(z)g(z)\,dA(z) & = (2\pi)^2  \int_{\R^2} \psi(\tau z)f(z)g(z) \,dA(z) \\*[7pt]
& =  \int_{\R^2} \widehat{(\psi(\tau\, \cdot)\, f)}(\xi) 
\,\overline{\widehat{g}(\xi)}\, d A(\xi) \\*[7pt]
& =  \int_{\R^2} (\widehat{\psi}_\tau \star \widehat{f}\,)(\xi)\, \overline{\widehat{g}(\xi)}\, d A(\xi), 
\end{split}
\end{equation}
where $\widehat{\psi}_\tau(z):= (2\pi \tau)^{-2} \widehat{\psi}(z/\tau)$. We have $\widehat\psi\in\mathcal S\subset L^1(\R^2)$ and 
${(2 \pi)}^{-2} \,\int_{\R^2} \,\widehat\psi(\xi)\,d\xi=\psi(0)=1.$  Hence the family $(\widehat{\psi}_\tau)_\tau$ is an approximate identity.  Being $\widehat f\in L^1(\R^2)$,  we conclude that $\widehat{\psi}_\tau \star \widehat{f}$ converges to $\widehat f$
in $L^1(\R^2)$, as $\tau\to0$. 

Since $\widehat{g}\in L^{\infty}(\R^2)$, we deduce that
$$
\lim_{\tau\to0} \int_{\R^2} (\widehat{\psi}_\tau \star \widehat{f}\,)(\xi)\, \overline{\widehat{g}(\xi)}\, d A(\xi)
= \int_{\R^2} \widehat{f}(\xi)\, \overline{\widehat{g}(\xi)}\, d A(\xi),
$$ 
which, together with \eqref{lime}, proves \eqref{parep}.

We now let $\e \rightarrow 0$ in \eqref{parep}. For the right-hand side we remark that
for every~$\xi\in\R^2$
$$
\widehat{\varphi_\e}(\xi) = \widehat{\varphi}(\e \xi) \to \widehat{\varphi}(0) =1,
$$
as $\e\to0$, and that
$\|\widehat{\varphi_\e}\|_{L^\infty}\le \|\varphi \|_{L^1} =1$
for each $\e >0$.
Therefore, by the Dominated Convergence Theorem, we have
\begin{equation*}\label{trickJV}
\int_{\R^2}  \widehat{W}_\alpha(\xi) |\widehat{\nu_\e}(\xi)|^2 \, dA(\xi) = \int_{\R^2}  \widehat{W}_\alpha(\xi) |\widehat{\nu}(\xi)|^2 
|\widehat{\varphi_\e}(\xi) |^2 \, d A(\xi) \ \to \ \int_{\R^2}  \widehat{W}_\alpha(\xi) |\widehat{\nu}(\xi)|^2  \, d A(\xi),
\end{equation*}
as $\e\to0$, even if the right-hand side is infinite.

To deal with the left-hand side of \eqref{parep}, we take $R>1$ such that the support of $\nu$ is contained in the
disc centred at the origin and with radius $R/4.$  Then
\begin{equation*}\label{}
|W_\alpha(z)|  \le \log\frac{R}{|z|}+|\alpha|+\log R, 	\quad |z|<R.
\end{equation*}
Set $\beta= |\alpha|+\log R,$  so that we obtain, for $\ep < \frac{R}{2}$, 
\begin{equation}\label{modW}
(|W_\alpha| \star \varphi_\e)(z) \le \left(\log\frac{R}{|z|} \star  \varphi_\e\right)(z)+\beta \le \log\frac{R}{|z|} +\beta, \quad |z|<\frac{R}{2},
\end{equation}
invoking the fact that the function $-\log|z|$ is superharmonic and $\varphi$ is radial (one writes the convolution
in polar coordinates and then applies superharmonicity on each circle).  

Note that 
\begin{equation}\label{kereg}
(W_\alpha\star \varphi_\e)(z) \xrightarrow{\e \to 0} W_\alpha(z), \quad z \in \R^2,
\end{equation}
because $W_\alpha$ is continuous as a function with values into $[0,+\infty]$. 

We claim  that $\nu$ has finite $I_\alpha$ energy \eqref{finen-0}, which can be translated into the condition 
\begin{equation}\label{finen}
\iint \log\frac{1}{|z-w|} \, d|\nu|(z) \, d|\nu|(w) < \infty.
\end{equation}

Recall that $\nu=\mu_1-\mu_2$  with $\mu_1$  and $\mu_2$ 
probability measures with finite energy and compact support. Thus $|\nu| \le \mu_1+\mu_2$ and \eqref{finen} is a consequence of
\begin{equation}\label{dos}
\begin{split}
2 \iint \log\frac{1}{|z-w|}\,d\mu_1(z) \,d\mu_2(w) &\le   \iint \log\frac{1}{|z-w|}  \,d\mu_1(z) \,d\mu_1(w)\\
&\quad+ \iint \log\frac{1}{|z-w|} \,d\mu_2(z) \,d\mu_2(w).
\end{split}
\end{equation}
The inequality above is clearly equivalent to
\begin{equation}\label{pos}
 \iint  \log\frac{1}{|z-w|} \,d(\mu_1-\mu_2)(z) \, d(\mu_1-\mu_2)(w) \geq 0,
\end{equation}
which is proven in \cite[Lemma 1.8, p.29]{ST} and \cite[Theorem 1.16, p.80]{L}. A more self-contained proof of \eqref{finen} is presented in Remark \ref{self-contained}. 

Combining \eqref{modW}, \eqref{kereg}, \eqref{finen} and the Dominated Convergence Theorem, we get
\begin{equation}\label{conv}
\iint_{\mathbb{R}^2\times \mathbb{R}^2}(W_\alpha\star \varphi_\e)(z-w) \, d\nu(z)\,d\nu(w) \ \to \ 
\iint_{\mathbb{R}^2\times \mathbb{R}^2} W_\alpha(z-w)  \, d\nu(z)\, d\nu(w),
\end{equation}
as $\e \to 0$, even if the right-hand side is infinite.

We now go back to the left-hand side of \eqref{parep} and observe that
\begin{equation*}
\int_{\R^2} (W_\alpha\star \nu_\e)(z)\nu_\e(z) \, dA(z) = \iint_{\mathbb{R}^2\times \mathbb{R}^2}(W_\alpha\star \varphi_\e\star \varphi_\e)(z-w) \, d\nu(z)\,d\nu(w).
\end{equation*}
Note that $(\varphi_\e \star \varphi_\e)(z) = \e^{-2}(\varphi \star \varphi)(z/\e)$ and that $\varphi \star \varphi$ inherits the properties of $\varphi$:
it is radial, belongs to $C^\infty_c(\R^2)$, and $\int_{\R^2} (\varphi \star \varphi)(z)\,dz =1$. 
Therefore, \eqref{conv} holds with $\varphi_\e$ replaced by $\varphi_\e \star \varphi_\e$. 
\end{proof}

\begin{remark*}[Self-contained proof of \eqref{finen}] \label{self-contained}
We provide now an alternative proof of \eqref{dos} (and therefore of \eqref{finen} and \eqref{pos}), based on Fatou's lemma
and on the superharmonicity of $-\log|z|$. 
To simplify the writing set $L(z)=\log(1/|z|).$ The mutual logarithmic energy of 
$\mu_1$ and $\mu_2$
can be written as
\begin{equation*}\label{}
 \int (L\star \mu_1) \, d\mu_2 =\left(L\star \mu_1 \star \widetilde{\mu_2}\right)(0),
\end{equation*}
where, given a positive Radon measure $\mu,$ $\widetilde{\mu}$ stands for the measure whose action on a test function $\varphi$ is 
$$
\int \varphi(-z)\, d\mu(z).
$$
Let $\varphi_\ep$ be an approximation of the identity as in the proof of the Lemma. Then $\varphi_\ep \star \varphi_\ep = (\varphi\star\varphi)_\ep$.
By Fatou's lemma
\begin{equation*}\label{}
\begin{split}
\left(L\star \mu_1 \star \tilde{\mu_2}\right)(0) & = \left(\left(\lim_{\ep \to 0} L\star \varphi_\ep \star \varphi_\ep\right) \star
\mu_1 \star \widetilde{\mu_2}\right)(0)\\*[8pt]
& \le \liminf_{\ep\to 0} \left( L\star \varphi_\ep \star \varphi_\ep  \star
\mu_1 \star \widetilde{\mu_2}\right)(0) \\*[8pt]
&= \liminf_{\ep\to 0} \left( L\star (\varphi_\ep \star
\mu_1) \star \widetilde{(\varphi_\ep \star \mu_2)} \right)(0) \\*[8pt]
& = \liminf_{\ep\to 0} \iint L(z-w)\, \mu_{1\ep}(w) \,dA(w) \mu_{2\ep}(z)\,dA(z),
\end{split}
\end{equation*}
where $\mu_{i\e} = \varphi_\e\star \mu_i$, for $i=1,2$. In view of \eqref{parep} for $\alpha=0$ we have
$$
\int_{\R^2} \left(L \star \nu_\e \right)(z) \, \nu_\e(z) \,dA(z) \ge 0,
$$
and so
\begin{equation*}\label{}
\begin{split}
2 \iint L(z-w)\, \mu_{1\ep}(w) \,dA(w) \mu_{2\ep}(z)\,dA(z) & \le \iint L(z-w) \mu_{1\ep}(z) \mu_{1\ep}(w) 
\,dA(z)\,dA(w)
\\*[8pt] & +  \iint L(z-w) \mu_{2\ep}(z) \mu_{2\ep}(w)\, dA(z)\,dA(w).
\end{split}
\end{equation*}
It remains to estimate the energies in the right-hand side of the previous inequality. Let $\mu$ a positive compactly supported
Radon measure. We then have
\begin{equation*}\label{}
\begin{split}
\iint L(z-w) \mu_{\ep}(z) \mu_{\ep}(w) \,dA(z)\,dA(w) & =\left(L\star \mu \star \widetilde{\mu}\star (\varphi \star \varphi)_\ep 
 \right)(0) \\*[8pt]
 & \le \left(L\star\mu\star \widetilde{\mu}\right)(0),
\end{split}
\end{equation*}
appealing to the superharmonicity of $L\star\mu\star \widetilde{\mu}$ and the fact that $\varphi\star\varphi$ is non-negative, 
 radial and with integral equal to $1.$ Therefore \eqref{dos} holds.
\end{remark*}

\section{The candidate ellipse}

In this section we show that the potential of the normalised characteristic function of the domain enclosed by the ellipse with horizontal semi-axis $\sqrt{1-\alpha}$ and vertical semi-axis $\sqrt{1+\alpha}$
satisfies the first Euler--Lagrange condition. We need to work with  a general ellipse with semi-axis $a$ and $b$ and the enclosed set 
\begin{equation*}\label{}
E=E(a,b)= \left\{(x,y) : \frac{x^2}{a^2}+\frac{y^2}{b^2} \le 1\right\}.
\end{equation*}
Let $P$ be the potential of the normalised characteristic function of $E$ defined as in \eqref{potmu}. The first Euler--Lagrange equation states that $P$ is
constant on $E,$ or equivalently, that its gradient is $0$ on $\mathring{E}.$  Expressing $x$ as $(z+\overline{z})/2$ and recalling that $\nabla= 2 \partial/\partial\, \ba$  one obtains
\begin{equation}\label{grapot}
\nabla P(z)=\left(-\frac{1}{\overline{z}} + \frac{\alpha}{2} \left(\frac{1}{z} -\frac{z}{\ba^2} \right)\right) \star 
 \frac{1}{|E|} \chi_E+z, \quad z \in \C.
\end{equation}
Note that the formula above, which holds in the sense of distributions, implies that $\nabla P$ is a continuous function, the first term being
the convolution of a locally integrable kernel with a bounded compactly supported function. Hence $P$ is of class $C^1$ in the whole plane.
To check \textbf{EL1} one has to compute explicitly on $E$ the potentials
\begin{equation*}\label{}
\frac{1}{\overline{z}} \star \chi_E \qquad \text{and} \qquad \frac{z}{\ba^2}  \star \chi_E.
\end{equation*}
Once we have these explicit formulas we will set the equation $\nabla P=0$ on $E$ and solve it for $a$ and $b.$

We start by computing the Cauchy potential of the characteristic function of $E,$ following \cite{HMV}.
Recall that $1/\pi z$ is the fundamental solution of the operator $\overline{\partial}= \partial/\partial\, \ba.$ Hence
\begin{equation*}\label{}
 \left(\frac{1}{\pi z}\star \chi_{E}\right)(z)= \ba+f(z), \quad z \in \mathring{E},
\end{equation*}
with $f$ holomorphic on $\mathring{E}.$ The function $f$ has to be chosen so that on the boundary of~$E$ the function $\ba+f(z)$
extends holomorphically to $\C \setminus E.$ Writing the equation of the ellipse in the variables $z$ and $\ba$ and solving for $\ba$
one obtains
\begin{equation*}\label{}
\ba=\lambda z + 2ab \,h(z), \quad z \in \partial E,\quad \lambda = \frac{a-b}{a+b},
\end{equation*} 
where $h$ is
\begin{equation*}\label{aich}
h(z)=\frac{1}{z+\sqrt{z^2+c^2}}, \qquad \text{with}\qquad c^2=b^2-a^2.
\end{equation*} 
The domain of the holomorphic function $h$ is the complement in the plane of the segment joining the foci of the ellipse. More explicitly, the domain of $h$ is $\C\setminus [-\sqrt{a^2-b^2}, \sqrt{a^2-b^2}]$ if $a\ge b$ and $\C\setminus [-i\sqrt{b^2-a^2}, i \sqrt{ b^2-a^2}]$
if $a \le b.$
Choosing $f(z) = -\lambda z$ one gets
\begin{equation}\label{cau}
 \left(\frac{1}{\pi z}\star \chi_{E}\right)(z) =
\begin{cases}
\ba-\lambda z, &z \in E, \\*[10pt]
2ab\,h(z), &z \in E^c.
\end{cases}
\end{equation}
The preceding identity follows from Liouville's theorem and the remark that both sides of \eqref{cau} are continuous functions on the plane, vanishing at $\infty,$
whose $\overline{\partial}$-derivative is the
characteristic function of $E.$
Taking conjugates
\begin{equation}\label{caucon}
 \left(\frac{1}{\pi \overline{z}}\star \chi_{E}\right)(z) =
\begin{cases}
z-\lambda \overline{z}, &z \in E, \\*[10pt]
2ab\,h(\overline{z}), &z \in E^c.
\end{cases}
\end{equation}
Now we reduce the computation of 
\begin{equation}\label{zibzi2}
 \frac{1}{\pi} \frac{z}{\bar{z}^2} \star \chi_E
\end{equation}
to \eqref{caucon}, by remarking that
\begin{equation*}\label{zezb2}
  \frac{1}{\pi} \frac{z}{\bar{z}^2} \star \chi_E =  (-\overline{\partial}) \left( \frac{1}{\pi} \frac{z}{\bar{z}} \star \chi_E \right),
\end{equation*}
and
\begin{equation*}\label{de+}
 \partial \left( \frac{1}{\pi} \frac{z}{\bar{z}} \star \chi_E\right)=  \frac{1}{\pi} \frac{1}{\bar{z}} \star \chi_E,
\end{equation*}
where we denote by $\partial = \frac{1}{2}\, \left(\frac{\partial}{\partial x}- i \frac{\partial}{\partial y}\right)$ the derivative with respect to the variable $z.$
Hence to compute~\eqref{zibzi2} one has to find a bounded primitive in $z$ of \eqref{caucon} and then
take $-\overline{\partial}.$  A primitive in $z$ of~\eqref{caucon} in $\mathring{E} \cup (\C\setminus E)$ is
\begin{equation}\label{primit}
 \frac{1}{2} (z-\lambda \overline{z})^2 \chi_E(z)+ \left(2abh(\overline{z}) H(z) +\varphi(\overline{z})\right) \chi_{E^c}(z),
\end{equation}
where $H(z)= z-\lambda \overline{z}-2abh(\ba)$ and $\varphi(z)$ is holomorphic on $E^c.$ Since on $\partial E$ 
$$
 \frac{1}{2} (z-\lambda \overline{z})^2 =  \frac{1}{2} (2abh(\ba))^2 
$$
we choose $\varphi(z) = \frac{1}{2} (2abh(z))^2 $ so that the function in \eqref{primit} is continuous and bounded on~$\C.$ 
The function in \eqref{primit} and $\frac{1}{\pi}  \frac{z}{\bar{z}} \star \chi_E$  are bounded primitives in $z$ of  
$\frac{1}{\pi} \frac{1}{\bar{z}} \star \chi_E,$ and so the difference is a bounded function on $\C$ annihilated by the operator
$\partial.$ This means that the conjugate function is a bounded entire function.
By Liouville's Theorem there is a constant $C$ such that
\begin{equation}\label{potzizib} 
\frac{1}{\pi}  \frac{z}{\bar{z}} \star \chi_E=
\frac{1}{2} (z-\lambda \overline{z})^2 \chi_E(z)+ \left(2abh(\overline{z}) H(z) + \frac{1}{2} (2abh(\ba))^2 \right) \chi_{E^c}(z)+C.
\end{equation}
It can be readily checked, examining the expansion at $\infty,$  that $C=\lambda ab,$  but this precise value is not important here. 
Taking $-\overline{\partial}$ in \eqref{potzizib} we get
\begin{equation}\label{zizib2}
\left(\frac{1}{\pi} \frac{z}{\overline{z}^2} \star \chi_E\right)(z) =\lambda (z-\lambda \overline{z}) , \quad z \in E.
\end{equation}
Indeed, one can obtain an explicit (although complicated) expression for the potential above also off $E,$ but precisely we want to show that it is not necessary to use it.

Plugging \eqref{cau}, \eqref{caucon} and \eqref{zizib2} in the formula \eqref{grapot} for the gradient of $P$ we get
\begin{equation*}\label{}
\nabla P (z) = \left( \frac{1}{ab} (-1-\alpha \lambda) +1\right) z + 
\frac{1}{ab} \left(\lambda +\frac{\alpha}{2}+ \frac{\alpha}{2} \lambda^2 \right) \ba, \quad z \in E.
\end{equation*}
Then $\nabla P$ vanishes on $E$ if and only if $a$ and $b$ are solutions of the system
\begin{equation*}\label{}
\begin{cases}
ab = 1+\alpha \lambda\\*[10pt]
\alpha \lambda^2+2 \lambda + \alpha =0.
\end{cases}
\end{equation*}
Solving the system yields $a=\sqrt{1-\alpha}$ and $b=\sqrt{1+\alpha} ,$ which provides an ellipse such that the potential 
of the normalised characteristic function of the enclosed domain satisfies the first Euler--Lagrange condition. 
This is the candidate ellipse. Our task in the next section is to show that the corresponding potential $P$ satisfies the second Euler--Lagrange condition.

\section{The second Euler--Lagrange condition}

Let $P$ stand for the potential of the normalised characteristic function of the domain $E$ enclosed by the candidate ellipse
found in the previous section.  We know that $P(z)=C_0$ for  $z \in  E$, and we have to prove that $P(z)\ge C_0$ for $z \notin E$.  
The proof proceeds in two steps. Our first task is to show the following result.
\begin{lemma*}
We have
\begin{equation}\label{lappotfr}
\lim_{E \not\ni w \to z}\Delta P(w) \ge \frac{2}{ab} (1-|\alpha|) >0, \quad z \in \partial E. 
\end{equation}
\end{lemma*}
\begin{proof}
 Taking $2 \partial$ in \eqref{grapot} one obtains
\begin{equation}\label{lappot}
\Delta P = -2\pi \frac{\chi_E}{|E|}-\alpha \;\pv \left(2 \re\left(\frac{1}{z^2}\right)\right)\star \frac{\chi_E}{|E|}+2.
\end{equation}
The jump of a function $f$ defined on $\mathring E \cup E^c$ at a point $z \in \partial E$ is
\begin{equation*}\label{}
\operatorname{jump}(f)(z) := \lim_{\mathring E \ni w \to z} f(w) -\lim_{E \not\ni w \to z} f(w).
\end{equation*}
Since $\Delta P=0$ on $\mathring E$ because of the first Euler--Lagrange equation, we have
\begin{equation*}\label{}
\lim_{E \not\ni w \to z} \Delta P(w)= - \operatorname{jump}(\Delta P)(z), \quad z \in \partial E.
\end{equation*}
To compute the jump of the Laplacian of $P$ we first need to express the principal value integral in \eqref{lappot} in a more suitable form. We have
\begin{equation*}\label{}
\operatorname{p.v.}\left(\frac{-2}{z^2}\right) \star \chi_E = \frac{1}{z}\star 2 \partial \chi_E = \frac{1}{z} \star (-\overline{n})\,d\sigma
=\frac{1}{iz} \star \overline{\tau^2} \,dz_{\partial E},
\end{equation*}
where $d\sigma$ is the length measure on $\partial E,$ $n$ the exterior unit normal vector and $\tau$ 
the unit tangent vector.

By \eqref{lappot} and the Plemelj's jump formula \eqref{ple} below we get
\begin{equation*}\label{}
\begin{split}
\lim_{E \not\ni w \to z} \Delta P(w) & = \frac{2\pi}{|E|}+ \frac{\alpha}{|E|} \re
\left(\operatorname{jump}\left(\frac{1}{iz} \star \overline{\tau^2} \,dz_{\partial E} \right)(z)\right)  \\*[10pt]
&= \frac{2}{ab} \left( 1- \alpha \, \re \left(\tau(z)^2 \right) \right)\\*[10pt]
& \ge \frac{2}{ab} (1-|\alpha|), \quad z \in \partial E.
\end{split}
\end{equation*}
\end{proof}
We turn now to the second step of the proof of the second Euler--Lagrange condition.
Recall that $P$ satisfies the first Euler--Lagrange condition, so that  $P(z) = C_0, \; z \in E.$  
We want to show that $P(z) \ge C_0, \; z \notin E,$ and for this we would like to apply the minimum principle.
Note that the term $|z|^2/2$ makes $P$ larger than $C_0$ at $\infty.$  But $P$ is not harmonic off $E,$ nor superharmonic, 
and thus the minimum principle cannot be applied directly to $P.$ Identity \eqref{lappot} shows that $P$ is biharmonic off E, 
that is, the Laplacian is harmonic, and for these functions there are some well-known manipulations that allow an application of
the minimum principle \cite{D}. We found an inspiration in that paper to devise the argument below.

Take $a \notin E$ and let $a_0 \in \partial E$ stand for the projection of $a$ into $E.$  Let $\vec{n}$ be the exterior unit normal 
vector to $\partial E$ at the point $a_0,$ so that $a$ belongs to the ray emanating from $a_0$ in the direction $\vec{n}.$
If we set
\begin{equation*}\label{}
h(z)=  \langle \nabla P(z), \vec{n}\rangle - \frac{1}{2} \Delta P(z) \langle z-a, \vec{n} \rangle, \quad z \notin E,
\end{equation*}
then we get
\begin{equation*}\label{}
\Delta h(z)=  \langle \nabla \Delta P(z), \vec{n} \rangle - \langle \nabla \Delta P(z), \nabla(\langle z-a,\vec{n} \rangle )=0, \quad z \notin E.
\end{equation*}
Since $P$ is of class $C^1$ on $\C$ the first Euler--Lagrange condition implies that $\nabla P(z)=0$ on $E.$ Hence
\begin{equation*}\label{}
\lim_{ E \not\ni w \to z} h(w)=  - \frac{1}{2} \lim_{ E \not\ni w \to z} \Delta P(w) \langle z-a,\vec{n} \rangle \ge 0, \quad z \in \partial E.
\end{equation*}
This follows on the one hand by the Lemma, which yields $\lim_{ E \not\ni w \to z} \Delta P(w) \ge 0, \; z \in \partial E,$ 
and on the other hand by the inequality
$\langle z-a,\vec{n} \rangle 
\le 0, \;  z \in \partial E,$ which is due to the convexity of $E.$ 

For $z $ large the behaviour of $h$ is controlled by the function one gets by replacing $P$ with the dominant term $|z|^2/2$ in the definition of $h.$ This function is 
\begin{equation*}\label{}
\langle z,\vec{n} \rangle - \langle z-a,\vec{n} \rangle = \langle a,\vec{n} \rangle >0,
\end{equation*}
which yields $h(z) > 0,$ for  $z$ large enough. Therefore, by the minimum principle, $h(z) \ge 0, \; z \notin E,$  and so $0\le h(a)=\langle \nabla P(a), \vec{n}\rangle .$
One concludes that $P$ is increasing along the ray starting at $a_0$ in the direction of $\vec{n}$ and so $P(a) \ge C_0.$

\section{Appendix: The Plemelj jump formula}

Let $\Gamma$ be a smooth Jordan curve enclosing a domain $D.$ Let $f$ be a smooth function on the curve. Define the Cauchy integral of $f$ as
\begin{equation*}
C(f)(z)=\frac{1}{2\pi i} \int_{\Gamma} \frac{f(\zeta)}{\zeta-z}\, d\zeta, \quad z \notin \Gamma.
\end{equation*}
The Plemelj formula (also called Plemelj--Sohotski) is the following:
\begin{equation}\label{ple}
\lim_{D \ni  z \to a}C(f)(z)-\lim_{\overline{D} \not\ni  z \to a}C(f)(z) = f(a), \quad a \in \Gamma.
\end{equation}
Note that the limits in the left-hand side of \eqref{ple} exist, because $\Gamma$ and $f$ are smooth. 

There is a much more general version of the formula in which $\Gamma$ is a rectifiable Jordan curve, $f$ is integrable with respect to the arc length measure,  the limits in the left-hand side are non-tangential and the identity holds a.e.\ with respect to the arc length measure on $\Gamma$. 

We now prove \eqref{ple} under the assumption that $\Gamma$ is rectifiable and $f$ is Lipschitz on $\Gamma.$ Appealing to a
well-known extension theorem we can further assume that $f$ is defined and Lipschitz on the whole plane.  Since the winding number of $\Gamma$  with respect to a point in $D$ is $1$ and with respect to a point off $\overline{D}$ is $0$, we have
\[
C(f)(z)=\frac{1}{2\pi i}\int_{\Gamma}\frac{f(\zeta)-f(z)}{\zeta-z}\,d\zeta+f(z),\quad
z\in D,
\]
and
\[
C(f)(z)=\frac{1}{2\pi i}\int_{\Gamma}\frac{f(\zeta)-f(z)}{\zeta-z}\,d\zeta,\quad z\notin \overline{D}.
\]
Taking limits as $z$ tends to $a$ from $D$ and from $\C\setminus \overline{D}$, we obtain
\begin{equation}\label{A}
\lim_{D \ni  z \to a} C(f)(z)=\frac{1}{2\pi i}\int_{\Gamma}\frac{f(\zeta)-f(a)}{\zeta-a}\,d\zeta+f(a),\quad a\in \Gamma,
\end{equation}
and
\begin{equation}\label{AA}
\lim_{\overline{D} \not \ni  z \to a} C(f)(z)=\frac{1}{2\pi i}\int_{\Gamma}\frac{f(\zeta)-f(a)}{\zeta-a}\,d\zeta, \quad a \in \Gamma,
\end{equation}
by the Dominated Convergence Theorem. Subtracting \eqref{AA} from \eqref{A} we get \eqref{ple}.

The interested reader may consult \cite{V} for a relation with the boundary Cauchy singular integral, 
defined in terms of principal values,  in a basic context as the one considered here. For more general results in higher
dimensions one can see \cite{HMT} and  \cite{T}.

\begin{acknowledgements} 
JM and JV are supported by
2017-SGR-395 (Generalitat de Cata\-lunya),
and MTM2016-75390 (Mineco). 
MGM acknowledges support by the Universit\`a di Pavia
through the 2017 Blue Sky Research Project ``Plasticity at different scales: micro to macro"
and by GNAMPA--INdAM.
LR is partly supported by GNAMPA--INdAM through Projects 2018 and 2019.
LS acknowledges support by the EPSRC Grant EP/N035631/1. 

The authors would like to thank the referee for a thorough reading of the paper and for valuable suggestions 
which have improved the exposition.
\end{acknowledgements}

\vspace{0.5cm}
{\small
\begin{tabular}{@{}l}
J.\ Mateu and J.\ Verdera \\ Departament de Matem\`{a}tiques, Universitat Aut\`{o}noma de Barcelona,\\
Barcelona Graduate School of Mathematics, Barcelona, Catalonia,\\
Centre de Recerca Matem\`atica\\
{\it E-mail:} {\tt mateu@mat.uab.cat} and {\tt jvm@mat.uab.cat}\\*[5pt]
M.\ G.\ Mora\\ 
 Dipartimento di Matematica, Universit\`a di Pavia, Italy\\
{\it E-mail:} {\tt mariagiovanna.mora@unipv.it}\\*[5pt]
L.\ Rondi\\ 
Dipartimento di Matematica, Universit\`a di Milano, Italy\\
{\it E-mail:} {\tt luca.rondi@unimi.it}\\*[5pt]
L.\ Scardia\\
Department of Mathematics, Heriot-Watt University, Edinburgh, United Kingdom\\
{\it E-mail:} {\tt l.scardia@hw.ac.uk}
\end{tabular}}
\end{document}